\documentclass[12pt; a4paper]{amsart} %\documentclass[选项表]{类}[版本号yyyy/mm/dd]
\usepackage{amscd, amsfonts, amsmath, amssymb, amsthm, fixmath, mathrsfs}
\usepackage[all]{xy} %\usepackage[选项表]{宏包}[版本号yyyy/mm/dd] 数学字体宏包:mathrsfs, amsfonts, amssymb
\makeindex

\theoremstyle{definition} %标题与编号为黑体, 正文为正常字体
\newtheorem{Unity}{Unity}[section] %\newtheorem{定理环境名}{标题}[主计数器名]
\newtheorem*{Definition*}{Definition} %\newtheorem*{定理环境名}[已定义定理环境名]{标题} 手动编号, 不自动编号
\newtheorem{Definition}[Unity]{Definition} %\newtheorem{定理环境名}[已定义定理环境名]{标题} 与当前环境共用同一个序号计数器

\theoremstyle{plain} %标题与编号为黑体, 正文为斜体
\newtheorem*{Theorem*}{Theorem}
\newtheorem{Theorem}[Unity]{Theorem}
\newtheorem{Proposition}[Unity]{Proposition}
\newtheorem{Corollary}[Unity]{Corollary}
\newtheorem{Lemma}[Unity]{Lemma}

\theoremstyle{remark} %标题与编号为斜体, 正文为正常字体
\newtheorem*{Remark*}{Remark}
\newtheorem{Remark}[Unity]{Remark}

\numberwithin{Unity}{section}%\numberwithin{计数器}{主计数器}

\newcommand{\N}{\mathbb{N}}
\newcommand{\Z}{\mathbb{Z}}
\newcommand{\p}{\mathfrak{p}}
\newcommand{\q}{\mathfrak{q}}
\newcommand{\m}{\mathfrak{m}}
\newcommand{\V}{\mathrm{V}}
\newcommand{\fd}{\mathrm{fd}}
\newcommand{\id}{\mathrm{id}}
\newcommand{\Img}{\mathrm{Im}}
\newcommand{\hgt}{\mathrm{ht}}
\newcommand{\Min}{\mathrm{Min}}
\newcommand{\Max}{\mathrm{Max}}
\newcommand{\Ann}{\mathrm{Ann}}
\newcommand{\Att}{\mathrm{Att}}
\newcommand{\Cos}{\mathrm{Cos}}
\newcommand{\Hom}{\mathrm{Hom}}
\newcommand{\Ext}{\mathrm{Ext}}
\newcommand{\Tor}{\mathrm{Tor}}
\newcommand{\Spec}{\mathrm{Spec}}
\newcommand{\Supp}{\mathrm{Supp}}
\newcommand{\Ndim}{\mathrm{Ndim}}
\newcommand{\Cdim}{\mathrm{Cdim}}
\newcommand{\coker}{\mathrm{coker}}
\newcommand{\depth}{\mathrm{depth}}
\newcommand{\cograde}{\mathrm{cograde}}

\begin{document}

\title{Vanishing Properties of Dual Bass numbers}
\author{Lingguang Li}
\address{Department of Mathematics, Tongji University, Shanghai, P. R. China\\ School of Mathematical Sciences, Fudan University, Shanghai, P. R. China}
\email{LG.Lee@amss.ac.cn}
\begin{abstract}
Let $R$ be a Noetherian ring, $M$ an Artinian $R$-module, $\p\in\Cos_RM$. Then $\cograde_{R_{\p}}\Hom_{R}(R_{\p},M)=\inf\{~i~|~\pi_{i}(\p,M)>0\}$ and
$$\pi_{i}(\p,M)>0\Rightarrow\cograde_{R_{\p}}\Hom_{R}(R_{\p},M)\leq i\leq\fd_{R_{\p}}\Hom_{R}(R_{\p},M),$$
where $\pi_{i}(\p,M)$ is the $i$-th dual Bass number of $M$ with respect to $\p$, the integer $\cograde_{R_{\p}}\Hom_{R}(R_{\p},M)$ is the common length of any maximal $\Hom_{R}(R_{\p},M)$-quasi co-regular sequence contained in $\p R_{\p}$, and $\fd_{R_{\p}}\Hom_{R}(R_{\p},M)$ is the flat dimension of $R_{\p}$-module $\Hom_{R}(R_{\p},M)$ (Theorem \ref{Thm:Main}). Besides, we also study the relations among cograde, co-dimension and flat dimension of co-localization module $\Hom_{R}(R_{\p},M)$.
\end{abstract}
\maketitle

\section{Introduction}

Let $R$ be a Noetherian ring, $M$ a $R$-module, $\p\in\Spec~R$. H. Bass \cite{Bass63} defined so called Bass numbers $\mu_{i}(\p,M)$ by
using the minimal injective resolution of $M$ for all integers $i\geq 0$, and proved that $\mu_{i}(\p,M)=\dim_{k(\p)}\Ext_{R_{\p}}^{i}(k(\p),M_{\p})$. E. Enochs and J. Z. Xu defined the dual Bass numbers $\pi_{i}(\p,M)$ by using the minimal flat resolution of $M$ for all $i\geq 0$ in \cite{EnochsXu97}, and showed that
$\pi_{i}(\p,M)=\dim_{k(\p)}\Tor^{R_{\p}}_{i}(k(\p),\Hom_{R}(R_{\p},M))$ for any cotorsion $R$-module $M$. Comparing the formulas of $\mu_{i}(\p,M)$ and $\pi_{i}(\p,M)$, we see that Ext is replaced by Tor and the localization $M_{\p}$ is replaced by co-localization $\Hom_{R}(R_{\p},M)$ respectively.

The Bass numbers and dual Bass numbers are homological invariants provide a powerful means in studying structures of certain important modules and rings. The vanishing properties of Bass numbers has important roles in studying modules. It is well known that if $M$ is a finitely generated $R$-module, $\p\in\Supp_{R}M$, then $\mu_{i}(\p,M)>0$ if and only if $\depth_{R_{\p}}M_{\p}\leq i\leq\id_{R_{\p}}M_{\p}$, where $\id_{R_{\p}}M_{\p}$ is the injective dimension of $R_{\p}$-module $M_{\p}$. However, we known little about the vanishing properties of dual Bass numbers.
In this paper, we obtain some vanishing properties of dual Bass numbers. Let $R$ be a Noetherian ring, $M$ an Artinian $R$-module, $\p\in\Cos_RM$. We show that $\cograde_{R_{\p}}\Hom_{R}(R_{\p},M)=\inf\{i|\pi_{i}(\p,M)>0\}$ and
$$\pi_{i}(\p,M)>0\Rightarrow\cograde_{R_{\p}}\Hom_{R}(R_{\p},M)\leq i\leq\fd_{R_{\p}}\Hom_{R}(R_{\p},M),$$
where $\cograde_{R_{\p}}\Hom_{R}(R_{\p},M)$ is the common length of any maximal $\Hom_{R}(R_{\p},M)$-quasi co-regular sequence contained in $\p R_{\p}$ and $\fd_{R_{\p}}\Hom_{R}(R_{\p},M)$ is the flat dimension of $R_{\p}$-module $\Hom_{R}(R_{\p},M)$ (Proposition \ref{Prop:CogBass}, Theorem \ref{Thm:Main}). Moreover, if $R$ is a $U$ ring (Definition \ref{Def:Uring}) and $\fd_{R_{\p}}\Hom_{R}(R_{\p},M)=t<\infty$, then $\pi_{t}(\p,M)>0$.

In addition, we also study the relations among cograde, co-dimension and flat dimension of $\Hom_{R}(R_{\p},M)$. Let $R$ be a $U$ ring, $M$ an Artinian $R$-module. Then
$$\cograde_{R_{\p}}\Hom_{R}(R_{\p},M)\leq\Cdim_{R_{\p}}\Hom_{R}(R_{\p},M)\leq\fd_{R_{\p}}\Hom_{R}(R_{\p},M),$$
where $\Cdim_{R_{\p}}\Hom_{R}(R_{\p},M)$ is the co-dimension of $R_{\p}$-module $\Hom_{R}(R_{\p},M)$.

\section{Preliminaries}

Let $R$ be a ring, $S\subseteq R$ a multiplicative set, and $M$ an $R$-module. The $R_S$-module $\Hom_{R}(R_S,M)$ is called the \emph{co-localization} of $M$ with respect to $S$ (\cite{MelSch95}). The \emph{co-support} of $M$ is defined by $\Cos_RM=\{~\p\in\Spec~R~|~\Hom_{R}(R_{\p},M)\neq 0~\}$. A subset $X$ of $\Spec~R$ is called a \emph{saturated subset} of $\Spec~R$, if $X$ satisfies one of the following conditions: $(i)$. $X=\emptyset$; $(ii)$. $X\neq\emptyset$, and $\V(\p)=\{\q\in\Spec~R~|~\p\subseteq\q\}\subseteq X$ for any $\p\in X$. Let $\p,\q\in\Spec~R$, and $\q\subseteq\p$. By Hom-Tensor adjunction we have
$$\Hom_{R_{\p}}(R_{\q},\Hom_{R}(R_{\p},M))\cong\Hom_{R}(R_{\q}\otimes_{R_{\p}}R_{\p},M)\cong\Hom_{R}(R_{\q},M).$$
This implies that the co-support of any $R$-module is a saturated subset of $\Spec~R$.

\begin{Definition}
Let $R$ be a ring, a representable $R$-module $M$ is said to be a \emph{strongly representable $R$-module}, if $(0:_MI)=\{x\in M~|~I\subseteq\Ann_R(x)\}$ is representable for any finitely generated ideal $I$ of $R$.
\end{Definition}

It is easy to see that every Artinian module is strongly representable. However, we will show that a strongly representable module may not be an Artinian module.

\begin{Lemma}\label{Lem:kercoker}
Let $R$ be a topological ring, $x_1,x_2,\cdots,x_n\in R$, and $M$ a linearly compact $R$-module. Then $(0:_{M}(x_1,x_2,\cdots,x_n))$, $(x_1,x_2,\cdots,x_n)M$, and $M/(x_1,x_2,\cdots,x_n)M$ are linearly compact $R$-modules.
\end{Lemma}

\begin{proof}
Since $\varphi: M\stackrel{x_1}{\rightarrow}M$ is a continuous $R$-module homomorphism, and $\{0\}$ is a closed submodule of $M$.
We have $\ker(\varphi)=(0:_{M}x_1)$ and $\Img(\varphi)=x_1M$ are both linearly compact $R$-module by \cite[Lemma 2.2]{CuongNhan02i}. Moreover, since $$(0:_{M}(x_1,x_2,\cdots,x_{s}))=\ker((0:_{M}(x_1,x_2,\cdots,x_{s-1}))\stackrel{x_{s}}{\rightarrow}(0:_{M}(x_1,x_2,\cdots,x_{s-1})))$$ for $s=1,2,\cdots,n$. Then the proposition is followed by using induction on $n$.
\end{proof}

\begin{Proposition}\label{Prop:DimCoLoc}
Let $R$ be a topological ring, $S\subseteq R$ a multiplicative set, $M$ a (resp. strongly) representable linearly compact $R$-module. Then
$\Hom_{R}(R_S,M)$ is a (resp. strongly) representable $R_S$-module. Moreover, the functor $\Hom_{R}(R_S,-)$ is an exact functor from the category of (resp. strongly) representable linearly compact $R$-modules to the category of (resp. strongly) representable $R_S$-modules.
\end{Proposition}

\begin{proof}
Let $M$ be a representable linearly compact $R$-module. By \cite[Corollary 3.3]{CuongNhan02i}, there exists a minimal secondary representation of $M=M_{1}+\cdots+M_{n}$ and $\p_{i}=\sqrt{\Ann_RM_{i}}$ such that all $M_{i}$ are linearly compact. We can assume that $S\cap\p_{i}=\emptyset$ $(1\leq i\leq m)$ and $S\cap\p_{i}\neq\emptyset$ $(m+1\leq i\leq n)$ for some integer $1\leq m\leq n$. Then, $$\Hom_{R}(R_S,M)=\Hom_{R}(R_S,M_{1})+\Hom_{R}(R_S,M_{2})+\cdots+\Hom_{R}(R_S,M_{m})$$
is a minimal secondary representation of $\Hom_{R}(R_S,M)$ as a $R$-module such that $\Hom_{R}(R_S,M_{i})$ are $\p_i$ secondary $R$-modules by \cite[Theorem 4.2]{CuongNhan02i}. It is easy to check that $\Hom_{R}(R_S,M_{i})$ is a $S^{-1}\p_{i}$ secondary $S^{-1}R$-module for $1\leq i\leq m$.

If $M$ is a strongly representable linearly compact $R$-module. Let $I=(\frac{x_1}{s_1},\cdots,\frac{x_n}{s_n})$ be a finitely generated ideal of $R_S$, where $x_i\in R, s_i\in S$ for $i=1,\cdots,n$. Then
$$(0:_{\tiny{\Hom_{R}(R_S,M)}}I)=(0:_{\tiny{\Hom_{R}(R_S,M)}}(\frac{x_1}{1},\cdots,\frac{x_n}{1}))=\Hom_{R}(R_S,0:_{M}(x_1,\cdots,x_n)).$$
By Lemma \ref{Lem:kercoker}, $\Hom_{R}(R_S,0:_{M}(x_1,x_2,\cdots,x_n))$ is a representable $R_S$-module.

Since co-localization preserves the exactness of short exact sequence of linearly compact modules by \cite[Corollary 2.5]{CuongNhan02i}. Thus the functor $\Hom_{R}(R_S,-)$ is an exact functor.
\end{proof}

\begin{Remark}
Let $R$ be a ring, $M$ an Artinian $R$-module. Then $M$ is a strongly representable $R$-module (In this case, we endow $R$ and $M$ with discrete topology). Let $\p\in\Spec~R$, it is well known that $\Hom_{R}(R_{\p},M)$ is almost never an Artinian $R_{\p}$-module. However, $\Hom_{R}(R_{\p},M)$ must be a strongly representable $R_{\p}$-module.
\end{Remark}

Let $R$ be a ring, $M$ a $R$-module. R. N. Roberts in \cite{Roberts75} introduced the Noetherian dimension of $M$, denoted by $\Ndim_RM$, which is defined inductively as follows: when $M=0$, put $\Ndim_RM=-1$. Then by induction, for an integer $d\geq 0$, we put $\Ndim_RM=d$, if $\Ndim_RM=d$ is false and for every ascending chain $M_0\subseteq M_1\subseteq\cdots$ of submodules of $M$, there exists a positive integer $n_0$ such that $\Ndim_R(M_{n+1}/M_n)<d$ for all $n>n_0$. Therefore $\Ndim_RM=0$ if and only if $M$ is a non-zero Noetherian $R$-module.
On the other hand, for any representable linearly compact module $M$ over some Noetherian topological ring $R$, $\Att_{R}M$, $\Cos_RM$, $\Ann_RM$ have same minimal elements by \cite[Corollary 4.3]{CuongNhan02i}. Thus
$$\sup\{~\dim R/\p~|~\p\in \Cos_RM\}=\sup\{~\dim R/\p~|~\p\in\Att_{R}M\}=\dim R/\Ann_RM.$$
Hence, we introduce the following definition.

\begin{Definition}
Let $R$ be a ring, and $M$ an $R$-module. The \emph{co-dimension} of $M$ is defined as the integer (possibly infinite) $$\Cdim_{R}M=\sup\{~\dim R/\p~|~\p\in\Cos_RM~\}.$$
\end{Definition}

We except that $\Ndim_RM=\Cdim_RM$ for any Artinian $R$-modules $M$. Unfortunately, this equality does not holds in general. N. T. Cuong and L. T. Nhan \cite[Example 4.1]{CuongNhan02ii} showed that exists an Artinian module $M$ over a Noetherian local ring $(R,\m)$ such that $\Ndim_RM<\Cdim_RM$. However, we have the following result.

\begin{Proposition}\label{Prop:Ndim-Cdim}
Let $R$ be a Noetherian ring, $M$ an Artinian $R$-module such that $\Ann_R(0:_{M}\p)=\p$ for any $\p\in\V(\Ann_RM)$. Then $\Ndim_RM=\Cdim_RM$.
\end{Proposition}
\begin{proof}
Since $\Cdim_RM=\dim R/\Ann_RM$ for any Artinian $R$-module $M$. It is enough to show that $\Ndim_RM=\dim R/\Ann_RM$.

By \cite[Proposition 2.4]{CuongNhan02ii}, we have the inequality $\Ndim_RM\leq\dim R/\Ann_RM$. On the other hand, for any ideal $\frak{a}\subsetneq R$, we have $(\frak{a}+\Ann_RM)\subseteq\Ann_R(0:_M\frak{a})$, and $\Ann_R(0:_M\frak{a})\subseteq\Ann_R(0:_{M}\p)=\p$ for any $(\frak{a}+\Ann_RM)\subseteq\p\in\Spec~R$ by hypothesis. Thus
$\sqrt{\frak{a}+\Ann_RM}=\sqrt{\Ann_R(0:_M\frak{a})}$. Let $\Ndim_RM=d$. Then there exist elements $x_1,\cdots,x_d\in J(M)=\bigcap\limits_{\m\in\Supp_RM}\m$ such that $(0:_M(x_1,\cdots,x_d))\neq 0$ have finite length by \cite[Theorem 4.1]{Tang96}. Let $\frak{a}=(x_1,\cdots,x_d)$. Then
$$0=\dim_R(0:_M(x_1,\cdots,x_d))=\dim_RR/((x_1,\cdots,x_d)+\Ann_RM)\geq \dim R/\Ann_RM-d.$$
Hence, $\Cdim_RM=\Ndim_RM$.
\end{proof}

\begin{Remark}
Let $R$ be a topological ring (not necessarily Noetherian), $S\subseteq R$ a multiplicative set, $M$ a representable linearly
compact $R$-module. Then by the proof of Proposition \ref{Prop:DimCoLoc}, we have
\begin{eqnarray*}
\Cdim_{S^{-1}R}\Hom_{R}(S^{-1}R,M)
&=&\sup\{\dim S^{-1}R/S^{-1}\p|\p\in\Cos_RM,\p\cap S=\emptyset\}\\
&=&\sup\{\dim S^{-1}R/S^{-1}\p|\p\in\Att_RM,~\p\cap S=\emptyset\}.
\end{eqnarray*}
\end{Remark}
Let $\p\in\Cos_RM$, we denote
$$\hgt_{M}\p =\sup\{~n~|~\p_{0}\subsetneq \p_{1}\subsetneq\cdots\subsetneq\p_{n}=\p,~\p_{i}\in\Cos_RM~\text{for}~i=0,1,\cdots,n~\}.$$
It is obvious that $\hgt_M\p=\Cdim_{R_{\p}}\Hom_{R}(R_{\p},M)$.

\section{Filter co-regular sequence and cograde}

\begin{Definition}
Let $R$ be a ring, $X\subseteq\Spec~R$ and $M$ an $R$-module. A sequence $x_1,x_2,\cdots,x_n\in R$ is called an \emph{$M$-filter co-regular sequence with respect to $X$}, if
$\Cos_R((0:_{M}(x_1,\cdots,x_{i-1}))/x_i(0:_{M}(x_1,\cdots,x_{i-1})))\subseteq X~\text{for}~i=1,2,\cdots,n$.
\end{Definition}

For any $R$-module $M$, any $M$-quasi co-regular sequence is a $M$-filter co-regular sequence with respect to $\emptyset$. Moreover, if $M$ is a linearly compact $R$-module, then the converse is also true by \cite[Corollary 4.3]{CuongNam01}.

\begin{Proposition}\label{Prop:Coregelement}
Let $R$ be a Noetherian topological ring, $X$ a saturated set of $\Spec~R$ and $M$ a representable linearly compact $R$-module. Then $x\in R$ is a $M$-filter co-regular element with respect to $X$ if and only if $x\in R-\bigcup\limits_{\tiny{\p\in\Att_RM-X}}\p$.
\end{Proposition}

\begin{proof}
Let $x\in R$ be a $M$-filter co-regular element with respect to $X$. Suppose that there exists a prime ideal $\p\in \Att_RM-X$ such that $x\in \p$. Then we have $\p\in \Att_RM\cap\V((x))=\Att_R(M/xM)$ by \cite[Theorem 4.5]{CuongNhan02i}. This contradicts to $\Att_R(M/xM)\subseteq\Cos_R(M/xM)\subseteq X$. Hence, $x\in R-\bigcup\limits_{\tiny{\p\in \Att_RM-X}}\p$.

Let $M=M_{1}+M_{2}+\cdots+M_{n}$ be a minimal secondary representation of $M$, where $M_{i}$ is a $\p_{i}$ secondary
$R$-module for $i=1,2,\cdots,n$, and $\Att_RM=\{\p_{1},\p_{2},\cdots,\p_{n}\}$. Let $x\in R-\bigcup\limits_{\tiny{\p\in\Att_RM-X}}\p$.
If $x\in R-\bigcup\limits_{\tiny{\p\in\Att_RM}}\p$, then $M=xM$. Thus $\Cos_R(M/xM)=\emptyset\subseteq X$. Hence, $x\in R$ is a $M$-filter co-regular element with respect to $X$. If there is some $1\leq m\leq n$, such that $x\in\p_{i}\in \Att_RM\cap X$ for $1\leq i\leq m$ and $x\in R-\p_{i}$ for $m+1\leq i\leq n$, then\\
$xM=x(M_{1}+M_{2}+\cdots+M_{n})=xM_{1}+xM_{2}+\cdots+xM_{m}+M_{m+1}\cdots+M_{n}$, and
$M/xM=(M_{1}+M_{2}+\cdots+M_{n})/(xM_{1}+xM_{2}+\cdots+xM_{m}+M_{m+1}\cdots+M_{n})$
$\cong(M_{1}+M_{2}+\cdots+M_{m})/((M_{1}+\cdots+M_{m})\cap(xM_{1}+\cdots+xM_{m}+M_{m+1}\cdots+M_{n})).$
Thus $\Att_{R}(M/xM)\subseteq\Att_{R}(M_{1}+\cdots+M_{m})=\{\p_{1},\cdots,\p_{m}\}$. Thus $\Cos_R(M/xM)\subseteq\bigcup\limits_{1\leq i\leq m}\V(\p_i)\subseteq X$.
\end{proof}

\begin{Corollary}\label{Cor:ExisCoregele}
Let $R$ be a Noetherian topological ring, and $X$ a saturated subset of $\Spec~R$. Let $I$ be an ideal of $R$ and $M$ a representable linearly compact $R$-module. Then $\Cos_R(M/IM)\subseteq X$ if and only if there is a $M$-filter co-regular element with respect to $X$ contained in $I$.
\end{Corollary}

\begin{proof}
Since $M/IM$ is a representable linearly compact $R$-module, then by \cite[Theorem 4.5]{CuongNhan02i} and Proposition \ref{Prop:Coregelement}, we have
\begin{eqnarray*}
\Cos_R(M/IM)\subseteq X&\Leftrightarrow&\Att_R(M/IM)\subseteq X\\
&\Leftrightarrow&\Att_RM\cap\V(I)\subseteq X\\
&\Leftrightarrow&I\nsubseteq\bigcup\limits_{\tiny{\p\in\Att_RM-X}}\p\\
&\Leftrightarrow&\text{$I$ contains a $M$-filter co-regular element with respect to $X$}.
\end{eqnarray*}
\end{proof}

\begin{Proposition}\label{Prop:CoQuasiReg}
Let $R$ be a topological ring, $X\subseteq\Spec~R$ and $M$ a linearly compact $R$-module. Then $x_1,x_2,\cdots,x_n\in R$ is a $M$-filter co-regular sequence with respect to $X$ if and only if $\frac{x_1}{1},\cdots,\frac{x_n}{1}\in R_{\p}$ is a $\Hom_{R}(R_{\p},M)$-quasi co-regular sequence for any $\p\in\Cos_RM-X$.
\end{Proposition}

\begin{proof}
For any $1\leq i\leq n$, consider the exact sequence of $R$-modules
$$(0:_{M}(x_1,\cdots,x_{i-1}))\stackrel{x_i}{\rightarrow}(0:_{M}(x_1,\cdots,x_{i-1}))\rightarrow\coker(x_i)\rightarrow0.$$

By \cite[Corollary 2.5]{CuongNhan02i}, we have an exact sequence of $R_{\p}$-modules
$$\Hom_{R}(R_{\p},0:_{M}(x_1,\cdots,x_{i-1}))\stackrel{\frac{x_i}{1}}{\rightarrow}\Hom_{R}(R_{\p},0:_{M}(x_1,\cdots,x_{i-1}))\rightarrow$$
$$\Hom_{R}(R_{\p},\coker(x_i))\rightarrow 0,$$
for any $\p\in\Spec~R-X$.

Since $\Hom_{R}(R_{\p},0:_{M}(x_1,\cdots,x_{i-1}))=(0:_{\tiny{\Hom_{R}(R_{\p},M)}}(\frac{x_1}{1},\cdots,\frac{x_{i-1}}{1}))$, we have the following equivalent statements:

The sequence $x_1,x_2,\cdots,x_n$ is an M-filter co-regular sequence with respect to $X$.
$\Leftrightarrow (0:_{\tiny{\Hom_{R}(R_{\p},M)}}(\frac{x_1}{1},\cdots,\frac{x_{i-1}}{1}))\stackrel{\frac{x_i}{1}}{\rightarrow}(0:_{\tiny{\Hom_{R}(R_{\p},M)}}(\frac{x_1}{1},\cdots,\frac{x_{i-1}}{1}))$ is surjective for any $\p\in\Cos_RM-X$, $i=1,2,\cdots,n$.
$\Leftrightarrow\frac{x_1}{1},\cdots,\frac{x_n}{1}\in R_{\p}$ is a $\Hom_{R}(R_{\p},M)$-quasi co-regular sequence for
any $\p\in\Cos_RM-X$.
\end{proof}

\begin{Definition}
Let $R$ be a ring, $X\subseteq \Spec~R$. A $R$-module sequence
$$M_{0}\stackrel{f_{1}}{\rightarrow}M_{1}\stackrel{f_{2}}{\rightarrow}\cdots\stackrel{f_{n-1}}{\rightarrow}M_{n-1}\stackrel{f_{n}}{\rightarrow}M_{n}$$
is called a \emph{quasi exact sequence of $R$-modules with respect to $X$}, if $f_{i+1}\circ f_{i}=0$, and $\Cos_R(\ker(f_{i+1})/\Img(f_{i}))\subseteq X$ for $i=1,2,\cdots,n-1$.
\end{Definition}

\begin{Lemma}\label{Lem:QuasiExact}
Let $R$ be a topological ring, $X\subseteq \Spec~R$ and a complex of $R$-modules
$$C_\bullet:~\cdots\stackrel{f_{i-2}}{\rightarrow}M_{i-1}\stackrel{f_{i-1}}{\rightarrow}M_{i}\stackrel{f_{i}}{\rightarrow}M_{i+1}\stackrel{f_{i+1}}{\rightarrow}\cdots$$
where $M_{i}$ are linearly compact $R$-module, $f_{i}$ are continuous homomorphisms for all integers $i\in\Z$. Then $C_\bullet$ is a quasi exact sequence of $R$-modules with respect to $X$ if and only if $\Hom_{R}(R_{\p},C_\bullet)$ is an exact sequence of $R_{\p}$-modules for any $\p\in \Spec~R-X$.
\end{Lemma}

\begin{proof}
Since the kernel and cokernel of a continuous homomorphism between linearly compact $R$-modules are linearly compact $R$-modules, and
the co-localization preserves the exactness of linearly compact $R$-modules (\cite[Corollary 2.5]{CuongNhan02i}), we get
\begin{eqnarray*}
\textrm{H}_{i}(\Hom_{R}(R_{\p},C_\bullet))&=&\ker(\Hom_{R}(R_{\p},d_i))/\Img(\Hom_{R}(R_{\p},d_{i-1}))\\
&\cong&\Hom_{R}(R_{\p},\ker(d_i))/\Hom_{R}(R_{\p},\Img(d_{i-1}))\\
&\cong&\Hom_{R}(R_{\p},\textrm{H}_{i}(C_\bullet)),~\text{for all}~i\in\Z~\text{and for any}~\p\in\Spec~R.
\end{eqnarray*}
Thus, $C.$ is a quasi exact sequence of $R$-modules with respect to $X$ if and only if $\Hom_{R}(R_{\p},\textrm{H}_{i}(C_\bullet))=0$, for all $i\in\Z$, $\p\in\Spec~R-X$. Equivalently, $\Hom_{R}(R_{\p},C_\bullet)$ is an exact sequence of $R_{\p}$-modules for all $\p\in\Spec~R-X$.
\end{proof}

The following theorem is the main result of this section.

\begin{Theorem}\label{Thm:Cograde}
Let $R$ be a Noetherian topological ring, $X$ a saturated subset of $\Spec~R$. Let $I$ be an ideal of $R$ and $M$ a strongly representable linearly compact $R$-module. Then for a given integer $n>0$ the following conditions are equivalent:
\begin{itemize}
\item[$(i)$.] $\Cos_R(\Tor_{i}^{R}(N,M))\subseteq X$ for all $0\leq i<n$ and for any finitely generated $R$-module N with
              $\Supp_RN\subseteq\V(I)$;
\item[$(ii)$.] $\Cos_R(\Tor_{i}^{R}(R/I,M))\subseteq X$ for all $0\leq i<n$;
\item[$(iii)$.] There exists a $M$-filter co-regular sequence with respect to $X$ of length n contained in I.
\end{itemize}
\end{Theorem}
\begin{proof}
$(i)\Rightarrow(ii)$. Obviously.

$(ii)\Rightarrow(iii)$. We use induction on $n$. For the case $n=1$ we have $R/I\otimes_{R}M=M/IM$ and $\Cos_R(M/IM)\subseteq X$. Then the assertion is followed by Corollary \ref{Cor:ExisCoregele}. Suppose that $n>1$, and the statement holds up to $n-1$. Let $x_1\in I$ be a $M$-filter co-regular element with respect to $X$, and $M_{1}=(0:_{M}x_1)$. Then $M_{1}$ is a strongly representable linearly compact $R$-module. Moreover, we get the following quasi exact sequence of $R$-modules with respect to $X$, $0\rightarrow M_{1}\rightarrow M\stackrel{x_1}{\rightarrow}M\rightarrow 0$. For $\p\in\Spec~R-X$, by Lemma \ref{Lem:QuasiExact} we have an exact sequence of $R_{\p}$-modules
$$0\rightarrow \Hom_{R}(R_{\p},M_{1})\rightarrow \Hom_{R}(R_{\p},M)
\stackrel{\frac{x_1}{1}}{\rightarrow}\Hom_{R}(R_{\p},M)\rightarrow 0.$$
From this we can get a long exact sequence of $R_{\p}$-modules
$$\cdots\rightarrow\Tor_{1}^{R_{\p}}(R_{\p}/IR_{\p},\Hom_{R}(R_{\p},M))
\rightarrow R_{\p}/IR_{\p}\otimes_{R_{\p}}\Hom_{R}(R_{\p},M_{1})\rightarrow$$
$$R_{\p}/IR_{\p}\otimes_{R_{\p}}\Hom_{R}(R_{\p},M)\stackrel{\frac{x_1}{1}}{\rightarrow} R_{p}/IR_{\p}\otimes_{R_{\p}}\Hom_{R}(R_{\p},M)\rightarrow 0.$$

Since $\Hom_{R}(R_{\p},\Tor_{i}^{R}(R/I,M))\cong\Tor_{i}^{R_{\p}}(R_{\p}/IR_{\p},\Hom_{R}(R_{\p},M))$ by the proof of \cite[Theorem 3.6]{CuongNam01}, we get $\Cos_R(\Tor_{i}^{R}(R/I,M_{1}))\subseteq X$ for all $0\leq i<n-1$. Then the statement is followed by the induction hypothesis.

$(iii)\Rightarrow(i)$. Let $x_1,x_2,\cdots,x_n\in I$ be a $M$-filter co-regular sequence with respect to $X$ of length $n$. Since
$\V(\Ann_RN)=\Supp_{R}N\subseteq\V(I)$, there exists $m>0$ such that $x_1^{m}\in\Ann_RN$. We use induction on $n$.

For the case $n=1$, we take co-localization to the homomorphism of $R$-modules $M\stackrel{x_1}{\rightarrow}M$, we get an exact sequence of $R_{\p}$-modules $$\Hom_{R}(R_{\p},M)\stackrel{\frac{x_1}{1}}{\rightarrow}\Hom_{R}(R_{\p},M)\rightarrow 0$$
for any $\p\in\Spec~R-X$. It follows that we have an exact sequence of $R_{\p}$-modules
$$N_{\p}\otimes_{R_{\p}}\Hom_{R}(R_{\p},M)\stackrel{\frac{x_1}{1}}{\rightarrow}N_{\p}\otimes_{R_{\p}}\Hom_{R}(R_{\p},M)\rightarrow 0.$$
Since $x_1^{m}\in\Ann_R(N)$, we have $(\frac{x_1}{1})^{m}\in\Ann_{R_{\p}}(N_{\p})$. Then
$$\Hom_{R}(R_{\p},N\otimes_{R}M)\cong N_{p}\otimes_{R_{\p}}\Hom_{R}(R_{\p},M)=0$$
for all $\p\in\Spec~R-X$. Hence $\Cos_R(N\otimes_{R}M)\subseteq X$.

Suppose that $n>1$, and the statement holds up to $n-1$. Set $M_{1}=(0:_{M}x_1)$. Consider the quasi exact sequence of $R$-modules
with respect to $X$ $$0\rightarrow M_{1}\rightarrow M\stackrel{x_1}{\rightarrow}M\rightarrow 0.$$
Then for any $\p\in\Spec~R-X$, by Lemma \ref{Lem:QuasiExact}, we have an exact sequence
$$0\rightarrow\Hom_{R}(R_{\p},M_{1})\rightarrow \Hom_{R}(R_{\p},M)\stackrel{\frac{x_1}{1}}{\rightarrow}\Hom_{R}(R_{\p},M)\rightarrow 0.$$
Therefore, for any integer $i\geq 0$, we have an exact sequence of $R_{\p}$-modules
$$\Tor_{i+1}^{R_{\p}}(N_{\p},\Hom_{R}(R_{\p},M))\stackrel{\frac{x_1}{1}}{\rightarrow}\Tor_{i+1}^{R_{\p}}(N_{\p},\Hom_{R}(R_{\p},M))$$
$$\rightarrow\Tor_{i}^{R_{\p}}(N_{\p},\Hom_{R}(R_{\p},M_{1})).$$

By induction hypothesis we have $\Tor_{i}^{R_{\p}}(N_{\p},\Hom_{R}(R_{\p},M_{1}))=0$ for any $\p\in\Spec~R-X$ and for any $0\leq i<n-1$. Moreover, $(\frac{x_1}{1})^{m}\in \Ann_{R_{\p}}(N_{\p})$, since $x_1^{m}\in\Ann_R(N)$. Then $\Tor_{i+1}^{R_{\p}}(N_{\p},\Hom_{R}(R_{\p},M))=0$ for any $\p\in\Spec~R-X$ and any $0\leq i<n-1$. Hence, $\Cos_R(\Tor_{i}^{R}(N,M))\subseteq X$ for any $0\leq i<n$.
\end{proof}

\begin{Definition}
Let $R$ be a Noetherian topological ring, $X$ a saturated subset of $\Spec~R$, $I$ an ideal of $R$ and $M$ a strongly representable linearly compact $R$-module. Then we define \emph{$M$-filter cograde with respect to $X$ contained in I} as the integer (possibly infinite)
$$\cograde_{X}(I,M)=\inf\{~i~|~\Cos_R(\Tor_{i}^{R}(R/I,M))\not\subseteq X~\}.$$
\end{Definition}

By Theorem \ref{Thm:Cograde}, it is easy to obtain the following facts:
\begin{itemize}
\item[$(1)$.] Any $M$-filter co-regular sequence with respect to $X$ contained in $I$ of finite length can be extended to a maximal one.
\item[$(2)$.] $\cograde_{X}(I,M)$ is possibly equal to $\infty$. e.g. if $X=\Cos_RM\cap\V(I)$.
\item[$(3)$.] $\cograde_{X}(I,M)$ is the common length of each maximal $M$-filter co-regular sequence with respect to $X$ contained in $I$.
\item[$(4)$.] Let $I^{'}$ be another ideal of $R$ such that $\sqrt{I}=\sqrt{I^{'}}$, then $\cograde_{X}(I,M)=\cograde_{X}(I^{'},M).$
\item[$(5)$.] Let $R$ be a Noetherian topological ring, $M$ a strongly representable linearly compact $R$-module. Then $\cograde_{\emptyset}(I,M)=\inf\{~i~|~\Tor_{i}^{R}(R/I,M)\neq 0\}$ equals to the length of any maximal $M$ co-regular sequence in $I$.
\item[$(6)$.] Let $(R,\m)$ be a Noetherian local ring, $I\subseteq R$ an ideal and $M$ an Artinian $R$-module. Then for any $i\geq 0$, $\Tor_{i}^{R}(R/I,M)$ is an Artinian $R$-module. Since $\Cos_RM\subseteq {\{\m\}}$ iff $M$ is a finite length $R$-module, we have $\cograde_{\{\m\}}(I,M)=\inf\{~i~|~\Tor_{i}^{R}(R/I,M)~\text{is not a finite length}~R-\text{module}\}$.
\end{itemize}

\section{Filter co-regular sequence and Quasi co-regular sequence}

Now, we will generalize some results of \cite{TangZakeri94} and \cite{Tang96}, and extend the theory of quasi co-regular sequence to some class of modules which are not necessarily Artinian.

\begin{Proposition}\label{Prop:Cogloc}
Let $R$ be a Noetherian topological ring , $M$ a strongly representable linearly compact $R$-module, $\p\in\Spec~R$, and $I$ an ideal of $R_{\p}$. Then for a given integer $n>0$ the following statements are equivalent:
\begin{itemize}
\item[$(1)$.] $\Tor_{i}^{R_{\p}}(N,\Hom_{R}(R_{\p},M))=0$ for all $0\leq i<n$ and for any finitely generated $R_{\p}$-module $N$ with $\Supp_{R_{\p}}N\subseteq\V(I)$.
\item[$(2)$.] $\Tor_{i}^{R_{\p}}(R_{\p}/I,\Hom_{R}(R_{\p},M))=0$ for any $0\leq i<n$.
\item[$(3)$.] There exists a $\Hom_{R}(R_{\p},M)$-quasi co-regular sequence of length $n$ contained in $I$.
\end{itemize}
\end{Proposition}

\begin{proof}
If $\p\not\in\Cos_RM$, this Proposition is obvious. Now we assume that $\p\in\Cos_RM$.

$(1)\Rightarrow(2)$. Obviously.

$(2)\Rightarrow(3)$. We use induction on $n$. For the case $n=1$, let $J\subseteq R$ be an ideal such that $I=JR_{\p}$. Since
$(R_{\p}/I)\otimes_{R_{\p}}\Hom_{R}(R_{\p},M)=\Hom_{R}(R_{\p},(R/J)\otimes_{R}M)=0$,
we have $\p\not\in\Cos_R(M/JM)$. Let $X=\{~\q~|~\q\in\Spec~R,~\q\not\subseteq\p\}$. Then $X$
is a saturated subset of $\Spec~R$. Obviously, $\Cos_R(M/JM)\subseteq X$ and there exists a $M$-filter co-regular element $x\in J$ with respect to $X$ by Corollary \ref{Cor:ExisCoregele}. Then by Proposition \ref{Prop:CoQuasiReg}, we know that $\frac{x}{1}\in I$ is a $\Hom_{R}(R_{\p},M)$-quasi co-regular element.

The proof of the case $n>1$ and the proof of $(3)\Rightarrow(1)$ is similar to the proof of Theorem \ref{Thm:Cograde}. We will not go into details.
\end{proof}

With the same notation of Proposition \ref{Prop:Cogloc}, we denote
$$\cograde_{R_{\p}}(I,\Hom_{R}(R_{\p},M))=\inf\{~i~|~\Tor_{i}^{R_{\p}}(R_{\p}/I,\Hom_{R}(R_{\p},M))\neq 0~\}.$$
In particular, if $I=\p R_{\p}$, we simply write $\cograde_{R_{\p}}\Hom_{R}(R_{\p},M)$. By Proposition \ref{Prop:Cogloc}, we notice that $\cograde_{R_{\p}}(I,\Hom_{R}(R_{\p},M))$ equals to the length of any maximal $\Hom_{R}(R_{\p},M)$-quasi co-regular sequence contained in $I$, and any $\Hom_{R}(R_{\p},M)$-quasi co-regular sequence contained in $I$ can be extended to a maximal one.

\begin{Proposition}
Let $R$ be a Noetherian topological ring, $X$ a saturated subset of $\Spec~R$, and $I\subsetneq R$ a proper ideal of $R$, and $M$ a strongly representable linearly compact $R$-module. Then
$$\cograde_{X}(I,M)=\inf\{\cograde_{R_{\p}}\Hom_{R}(R_{\p},M)|~\p\in\Cos_RM\cap\V(I)-X\}$$
\hspace*{70pt}$=\inf\{\cograde_{R_{\p}}(IR_{\p},\Hom_{R}(R_{\p},M))|~\p\in\Cos_RM\cap\V(I)-X\}$.
\end{Proposition}

\begin{proof}
If $\cograde_{X}(I,M)=\infty$, there is nothing to prove. So we assume that $\cograde_{X}(I,M)<\infty$. If $\Cos_RM\cap\V(I)\subseteq X$, then $\Cos_R(\Tor_{i}^{R}(R/I,M))\subseteq X$ for any $i\geq 0$. This induces a contradiction. Hence $\Cos_RM\cap\V(I)\not\subseteq X$.

$(i)$. Let $\p\in\Cos_RM\cap\V(I)-X$. Then we have
$$\cograde_{X}(I,M)\leq\cograde_{X}(\p,M)\leq\cograde_{R_{\p}}\Hom_{R}(R_{\p},M).$$

Let $x_1,x_2,\cdots,x_n$ be a maximal $M$-filter co-regular sequence with respect to $X$ contained in $I$. So there exists no $(0:_{M}(x_1,x_2,\cdots,x_n))$-filter co-regular element contained in $I$. By Proposition \ref{Prop:Coregelement}, there exists $\p\in\Att_R(0:_{M}(x_1,x_2,\cdots,x_n))-X$ such that $I\subseteq \p$. Then $\p\in\Cos_RM\cap\V(I)-X$, and
$$\p R_{\p}\in\Att_{R_{\p}}\Hom_{R}(R_{\p},0:_{M}(x_1,x_2,\cdots,x_n)).$$
Since $\Hom_{R}(R_{\p},0:_{M}(x_1,x_2,\cdots,x_n))\cong (0:_{\tiny{\Hom_{R}(R_{\p},M)}}(\frac{x_1}{1},\cdots,\frac{x_n}{1}))$. Thus $\frac{x_1}{1},\cdots,\frac{x_n}{1}\in R_{\p}$ is a maximal $\Hom_{R}(R_{\p},M)$ co-regular sequence. Hence,
$$\cograde_{X}(I,M)=\Min\{~\cograde_{R_{\p}}\Hom_{R}(R_{\p},M)~|~\p\in\Cos_RM\cap\V(I)-X\}.$$

$(ii)$. Assume that $\cograde_{X}(I,M)=n$. Then $\Tor_{i}^{R_{\p}}(R_{\p}/IR_{\p},\Hom_{R}(R_{\p},M))=0$ for any $0\leq i<n$ and for any $\p\in\Cos_RM\cap\V(I)-X$. Moreover, there exists $\q\in\Cos_RM\cap\V(I)-X$, such that $\Tor_{n}^{R_{\q}}(R_{\q}/IR_{\q},\Hom_{R}(R_{\p},M))\neq 0$. Hence
$$\cograde_{X}(I,M)=\inf\{\cograde_{R_{\p}}(IR_{\p},\Hom_{R}(R_{\p},M))|\p\in\Cos_RM\cap\V(I)-X\}.$$
\end{proof}

Let $R$ be Noetherian ring, $I$ a ideal of $R$, $N$ a finitely generated $R$-module. If $IN\neq N$, then the $\depth_R(I,N)<\infty$. The dual question for Artinian module $M$ is to ask when $\cograde_{X}(I,M)<\infty$. We will give some sufficient conditions for finiteness of $\cograde_{X}(I,M)$.

\begin{Lemma}\label{Lem:Ucond}
Let $R$ be a Noetherian topological ring , $M$ a strongly representable linearly compact $R$-module. Then the following conditions are equivalent:
\begin{itemize}
\item[$(i)$.] $\Ann_R(0:_{M}\p)=\p$ for any $\p\in \V(\Ann_RM)$.
\item[$(ii)$.] $\Cos_R(0:_{M}I)=\Cos_RM\cap\V(I)$ for any ideal $I$ of $R$.
\end{itemize}
\end{Lemma}

\begin{proof}
$(i)\Rightarrow(ii)$. Let $I\subseteq R$ be an ideal. By \cite[Corollary 4.3]{CuongNhan02i}, we have
$$\V(\Ann_R(0:_{M}I))=\Cos_R(0:_{M}I)\subseteq\Cos_RM.$$ Thus $\Cos_R(0:_{M}I)\subseteq\Cos_RM\cap\V(I)$.

For an ideal $I$ of $R$, if $\Cos_RM\cap\V(I)=\emptyset$, there is nothing to prove. Suppose that $\Cos_RM\cap\V(I)\neq\emptyset$. Let $\p\in\Cos_RM\cap\V(I)$. Since
$$\Cos_R(0:_{M}\p)=\V(\Ann_R(0:_{M}\p)),~\text{and}~\Ann_R(0:_{M}\p)=\p~\text{for any}~\p\in\V(\Ann_RM)$$
we have $\p\in\Cos_R(0:_{M}\p)$. Notice that
$$\Hom_{R}(R_{\p},0:_{M}I)\supseteq(0:_{\tiny{\Hom_{R}(R_{\p},M)}}\p R_{\p})=\Hom_{R}(R_{\p},0:_{M}\p)\neq 0$$
Hence $\Cos_RM\cap\V(I)\subseteq\Cos_R(0:_{M}I)$.

$(ii)\Rightarrow(i)$. Let $\p\in\V(\Ann_RM)$. Then $\V(\Ann_R(0:_{M}\p))=\Cos_RM\cap\V(\p)=\V(\p)$.
Moreover, since $\p\subseteq\Ann_R(0:_{M}\p)$, we have $\Ann_R(0:_{M}\p)=\p$.
\end{proof}

\begin{Proposition}\label{Prop:FiniteCograde}
Let $R$ be a Noetherian topological ring, $X$ a saturated subset of $\Spec~R$ and $I\subsetneq R$ a proper ideal of $R$. Let $M$ be a strongly representable linearly compact $R$-module such that $\Ann_R(0:_{M}\p)=\p~\text{for any}~\p\in\V(\Ann_RM)$.
Then $\cograde_{X}(I,M)<\infty$ if and only if $\Cos_RM\cap\V(I)\not\subseteq X$.
\end{Proposition}
\begin{proof}
$\Rightarrow$. If $\Cos_RM\cap\V(I)\subseteq X$, then $\Cos_R(\Tor_{i}^{R}(R/I,M))\subseteq X$ for any $i\geq 0$, this contradicts to $\cograde_{X}(I,M)<\infty$. Hence $\Cos_RM\cap\V(I)\nsubseteq X$.

$\Leftarrow$. For $\p\in\Cos_RM\cap\V(I)-X$. Then by Lemma \ref{Lem:Ucond} we have
$$(0:_{\tiny{\Hom_{R}(R_{\p},M)}}IR_{\p})=\Hom_{R}(R_{\p},0:_{M}I)\neq 0.$$
Let $\frac{x_1}{s_1},\cdots,\frac{x_r}{s_r}\in IR_{\p}$ be a $\Hom_{R}(R_{\p},M)$ co-regular sequence. Then
$$0:_{\tiny{\Hom_{R}(R_{\p},M)}}(\frac{x_1}{s_1})\supsetneq 0:_{\tiny{\Hom_{R}(R_{\p},M)}}(\frac{x_1}{s_1},\frac{x_2}{s_2})
\supsetneq 0:_{\tiny{\Hom_{R}(R_{\p},M)}}(\frac{x_1}{s_1},\frac{x_2}{s_2},\cdots,\frac{x_r}{s_r}).$$
Otherwise, there must exist some integer $i$, $1\leq i\leq r$ such that
$0:_{\tiny{\Hom_{R}(R_{\p},M)}}(\frac{x_1}{s_1},\cdots,\frac{x_i}{s_i})=0$.
This contradicts to $0\neq 0:_{\tiny{\Hom_{R}(R_{\p},M)}}IR_{\p}\subseteq 0:_{\tiny{\Hom_{R}(R_{\p},M)}}(\frac{x_1}{s_1},\cdots,\frac{x_i}{s_i})$. It follows
that $(\frac{x_1}{s_1})\subsetneq(\frac{x_1}{s_1},\frac{x_2}{s_2})\subsetneq\cdots \subsetneq(\frac{x_1}{s_1},\frac{x_2}{s_2},\cdots,\frac{x_r}{s_r})$. Since $R_{\p}$ is a Noetherian ring, therefore every $\Hom_{R}(R_{\p},M)$ co-regular sequence in $IR_{\p}$ must be of finite length. Hence there exists an integer $n\geq 0$, such that $$\Tor_{n}^{R_{\p}}(R_{\p}/IR_{\p},\Hom_{R}(R_{\p},M))\neq 0.$$ Thus $\Cos_R(\Tor_{n}^{R}(R/I,M))\nsubseteq X$. Hence $\cograde_{X}(I,M)<\infty$.
\end{proof}

\begin{Remark}
By the proof of Proposition \ref{Prop:FiniteCograde}, we note that if $R$ is a Noetherian  topological ring, and $M$ is a strongly representable linearly compact $R$-module such that $\Ann_R(0:_{M}\p)=\p$ for any $\p\in\V(\Ann_RM)$, then $\cograde_{R_{\p}}\Hom_{R}(R_{\p},M)<\infty$ for any $\p\in\Cos_RM$.
\end{Remark}

N. T. Cuong, N. T. Dung and L. T. Nhan given an example to show that the equivalent conditions of Lemma \ref{Lem:Ucond} may not true for general Artinian modules (\cite{CDN07}). However, they also given some sufficient conditions for that condition.
\begin{Proposition}\cite[Proposition 2.1]{CDN07}
Let $(R,{\frak m})$ be a Noetherian local ring, $M$ an Artinian $R$-module. If one of the following cases happens:
\begin{itemize}
\item[$(1)$.] $R$ is complete with respect to $\m$-adic topology.
\item[$(1)$.] $M$ contain a submodule which is isomorphic to the injective hull of $R/{\frak m}$.
\end{itemize}
Then $\Ann_R(0:_{M}\p)=\p$ for any $\p\in\V(\Ann_RM)$.
\end{Proposition}

\section{Vanishing Theorem of Dual Bass numbers}

Let $R$ be a Noetherian ring, $M$ an $R$-module. A minimal flat resolution of $M$ is an exact sequence of $R$-modules
$$\cdots\stackrel{d_{i+1}}{\rightarrow}F_{i}\stackrel{d_i}{\rightarrow}F_{i-1}\stackrel{d_{i-1}}{\rightarrow}\cdots\stackrel{d_1}{\rightarrow}F_{0}\stackrel{d_0}{\rightarrow}M\rightarrow 0,$$
such that for each $i\geq 0$, $F_{i}$ is a flat cover of $\Img(d_{i})$. The minimal flat resolution of $M$ exists and uniquely determined up to isomorphism. Moreover, for any integer $n$, $\fd_{R}M\leq n$ if and only if $F_{k}=0$ for any $k>n$. In this case, $F_{i}$ is a flat cotorsion $R$-module for any $i>0$, but $F_{0}$ may not be cotorsion in general. E. Enochs \cite{Enochs84} proved that $F_{i}$ is uniquely represented as a product $F_{i}=\prod\limits_{\p\in\tiny\Spec~R}T^{i}_{\p}$, where $T^{i}_{\p}$ is the completion of a free $R_{\p}$-module with respect to the $\p R_{\p}$-adic topology. In \cite{EnochsXu97}, $\pi_{i}(\p,M)$ is defined to be the cardinality of the base of a free $R_{\p}$-module whose completion is $T^{i}_{\p}$ for any $i>0$. On the other hand, $\pi_{0}(\p,M)$ is defined similarly by using the pure injective envelope $PE(F_{0})$ instead of $F_{0}$ itself. We call the $\pi_{i}(\p,M)$ the $i$-th \emph{dual Bass number} of $M$ with respect to $\p$. E. Enochs and J. Z. Xu \cite[Theorem 2.2]{EnochsXu97} showed that for any $R$-module $M$ over Noetherian ring $R$, there exists cotorsion $R$-module $E$ such that $\pi_{i}(\p,M)=\pi_{i}(\p,E)$ for any $i\geq 0$. The main results of \cite{EnochsXu97} is the following: Let $R$ be a Noetherian ring, $M$ a cotorsion $R$-module. Then $\pi_{i}(\p,M)=\dim_{k(\p)}\Tor^{R_{\p}}_{i}(k(\p),\Hom_{R}(R_{\p},M))$, for any $i\geq 0$.

By the proof of \cite[Theorem 2.2]{EnochsXu97}, we could get following Lemma.

\begin{Lemma}\label{Lem:MinFlatRe}
Let $R$ be a Noetherian ring, $\p,\q\in\Spec~R$ with $\q\subseteq\p$, and $M$ a cotorsion $R$-module. Let the following exact sequence
$$\cdots\stackrel{d_{i+1}}{\rightarrow}F_{i}\stackrel{d_i}{\rightarrow}F_{i-1}\stackrel{d_{i-1}}{\rightarrow}\cdots\stackrel{d_1}{\rightarrow}F_{0}\stackrel{d_0}{\rightarrow}M\rightarrow 0$$
be a minimal flat resolution of $M$. Then
$$\cdots\rightarrow\Hom_{R}(R_{\p},F_{i+1}) \rightarrow\Hom_{R}(R_{\p},F_{i})\rightarrow\Hom_{R}(R_{\p},F_{i-1})\rightarrow$$
$$\cdots\rightarrow\Hom_{R}(R_{\p},F_{1}) \rightarrow\Hom_{R}(R_{\p},F_{0})\rightarrow\Hom_{R}(R_{\p},M)\rightarrow 0$$
is a minimal flat resolution of $\Hom_{R}(R_{\p},M)$ as an $R_{\p}$-module, and if $F_{i}=\prod\limits_{\q\in\tiny{\Spec R}}T^{i}_{\q}$, then $\Hom_{R}(R_{\p},F_{i})\cong\prod\limits_{\q\subseteq \p\atop \q\in\tiny{\Spec R}}T^{i}_{\q}$ is a flat cotorsion $R_{\p}$-module. In other words, $\pi_{i}(\q,M)=\pi_{i}({\q}R_{\p},\Hom_{R}(R_{\p},M))$ for any $i\geq 0$.
\end{Lemma}

\begin{Proposition}\label{Prop:CogBass}
Let $R$ be a Noetherian ring, $M$ an Artinian $R$-module. Then for any $\p\in\Cos_RM$, $\cograde_{R_{\p}}\Hom_{R}(R_{\p},M)=\inf\{~i~|~\pi_{i}(\p,M)>0\}$.
\end{Proposition}

\begin{proof}
Since any Artinian module is cotorsion (\cite[Theorem 2.8]{EnochsXu97}). Then the formula $\pi_{i}(\p,M)=\dim_{k(\p)}\Tor^{R_{\p}}_{i}(k(\p),\Hom_{R}(R_{\p},M))$ holds for any Artinian module. Hence, this Proposition follows from Proposition \ref{Prop:Cogloc}.
\end{proof}

\begin{Lemma}\label{Lem:SuppCot}
Let $R$ be a Noetherian ring, $F$ a flat cotorsion $R$-module. Then $$F\neq 0\Leftrightarrow\Cos_{R}F\neq\emptyset.$$
\end{Lemma}

\begin{proof} Suppose that $F\neq 0$. Assume that $F=\prod\limits_{\p\in\tiny{\Spec R}}T_{\p}$, where $T_{\p}$ is the completion of a free $R_{\p}$-module with respect to the $\p R_{\p}$-adic topology. Then there exists $\p\in\Spec~R$ such that $T_{\p}\neq 0$. Then we have $\Hom_{R}(R_{\p},T_{\p})\cong T_{\p}\neq 0$. Hence $\Hom_{R}(R_{\p},F)\neq 0$, and thus $\Cos_{R}F\neq\emptyset$. Conversely, it is obvious.
\end{proof}

\begin{Proposition}\label{Prop:fd}
Let $R$ be a Noetherian ring, $M$ a cotorsion $R$-module. Then:
\begin{itemize}
\item[$(1)$.] $\fd_{S^{-1}R}\Hom_{R}(S^{-1}R,M)\leq\fd_{R}M$ for any multiplicative set $S\subseteq R$.
\item[$(2)$.] $\fd_{R}M=\sup\{~\fd_{R_{\p}}\Hom_{R}(R_{\p},M)~|~\p\in\Spec~R~\}$\\
\hspace*{29pt}$=\sup\{~\fd_{R_{\m}}\Hom_{R}(R_{\m},M)~|~\m\in\Max~R~\}$.
\end{itemize}
\end{Proposition}

\begin{proof}
Let the following exact sequence of $R$-modules
$$\cdots\rightarrow F_{i}\rightarrow F_{i-1}\rightarrow\cdots\rightarrow F_{0}\rightarrow M\rightarrow 0$$
be a minimal flat resolution of $M$. Then by \cite[Theorem 2.7]{EnochsXu97}, we have a minimal flat resolution of $S^{-1}R$-module $\Hom_{R}(S^{-1}R,M)$
$$\cdots\rightarrow\Hom_{R}(S^{-1}R,F_{i+1})\rightarrow\Hom_{R}(S^{-1}R,F_{i})\rightarrow\Hom_{R}(S^{-1}R,F_{i-1})\rightarrow$$
$$\cdots\rightarrow\Hom_{R}(S^{-1}R,F_{1})\rightarrow\Hom_{R}(S^{-1}R,F_{0})\rightarrow\Hom_{R}(S^{-1}R,M)\rightarrow 0.$$
Then $(1)$ follows. By Lemma \ref{Lem:SuppCot}, $(2)$ also holds.
\end{proof}

Proposition \ref{Prop:fd} is somehow dual to \cite[Corollary 2.3]{Bass63}.

\begin{Definition}\label{Def:Uring}
A Noetherian ring $R$ is called a \emph{U ring}, if for any Artinian $R$-module $M$, $\Ann_R(0:_{M}\p)=\p$ for any $\p\in\V(\Ann_RM)$.
\end{Definition}

Notice that any complete Noetherian local rings is a $U$ ring by \cite[Proposition 2.1]{CDN07}. Let $M$ be an Artinian $R$-module over a $U$ ring. Then for any $\p\in\Cos_RM$,
$$\cograde_{R_{\p}}\Hom_{R}(R_{\p},M)<\infty,$$
because $(0:_{\tiny{\Hom_{R}(R_{\p},M))}}\p R_{\p})\cong\Hom_{R}(R_{\p},0:_{M}\p)\neq 0$ and any co-regular sequence in Noetherian ring must be of finite length.

\begin{Theorem}\label{Thm:BassPlus}
Let $R$ be a $U$ ring, $\p,\q\in\Spec~R$ such that $\q\subset\p,~\hgt(\p/\q)=1$, and $M$ an Artinian $R$-module. Then $\pi_{i}(\q,M)\neq 0\Rightarrow\pi_{i+1}(\p,M)\neq 0$.
\end{Theorem}

\begin{proof}
For $x\in\p-\q$, from the short exact sequence of $R_{\p}$-modules
$$0\rightarrow R_{\p}/{\q}R_{\p}\stackrel{\frac{x}{1}}{\rightarrow}R_{\p}/{\q}R_{\p}\rightarrow R_{\p}/({\q}R_{\p},\frac{x}{1})\rightarrow 0,$$
we get the following long exact sequences of $R_{\p}$-modules
$$\cdots\rightarrow\Tor_{i+1}^{R_{\p}}(R_{\p}/({\q}R_{\p},\frac{x}{1}),\Hom_{R}(R_{\p},M))\rightarrow\Tor_{i}^{R_{\p}}(R_{\p}/{\q}R_{\p},\Hom_{R}(R_{\p},M))\stackrel{\frac{x}{1}}{\rightarrow}$$
$$\Tor_{i}^{R_{\p}}(R_{\p}/{\q}R_{\p},\Hom_{R}(R_{\p},M))\rightarrow\Tor_{i}^{R_{\p}}(R_{\p}/({\q}R_{\p},\frac{x}{1}),\Hom_{R}(R_{\p},M))\rightarrow\cdots.$$
Since
\begin{eqnarray*}
\Hom_{R_{\p}}(R_{\q},\Tor_{i}^{R_{\p}}(R_{\p}/{\q}R_{\p},\Hom_{R}(R_{\p},M)))&\cong&\Hom_{R_{\p}}(R_{\q},\Hom_{R}(R_{\p},\Tor_{i}^{R}(R/\q,M)))\\
&\cong&\Hom_{R}(R_{\q},\Tor_{i}^{R}(R/\q,M))\\
&\cong&\Tor_{i}^{R_{\q}}(k(\q),\Hom_{R}(R_{\q},M))\neq 0,
\end{eqnarray*}
we have $\Tor_{i}^{R_{\p}}(R_{\p}/{\q}R_{\p},\Hom_{R}(R_{\p},M))\neq 0$. Thus $\p\in\Cos_R\Tor_{i}^{R}(R/\q,M)$. Notice that $\Tor_{i}^{R}(R/\q,M)$ is an Artinian $R$-module, then
$$\Cos_R(0:_{\tiny{\Tor_{i}^{R}(R/\q,M)}}x)=\Cos_R\Tor_{i}^{R}(R/\q,M)\cap V((x)).$$
Since $x\in \p$, we get $\p\in \Cos_R(0:_{\tiny{\Tor_{i}^{R}(R/\q,M)}}x)$. Hence
\begin{eqnarray*}
(0:_{\tiny{\Tor_{i}^{R_{\p}}(R_{\p}/{\q}R_{\p},\Hom_{R}(R_{\p},M))}}\frac{x}{1})
&\cong&(0:_{\tiny{\Hom_{R}(R_{\p},\Tor_{i}^{R}(R/\q,M))}}\frac{x}{1})\\
&\cong&\Hom_{R}(R_{\p},0:_{\tiny{\Tor_{i}^{R}(R/\q,M)}}x)\neq 0.
\end{eqnarray*}
Thus, $\Tor_{i+1}^{R_{\p}}(R_{\p}/({\q}R_{\p},\frac{x}{1}),\Hom_{R}(R_{\p},M))\neq 0$. Notice that $\hgt(\p/\q)=1$, therefore $R_{\p}/({\q}R_{\p},\frac{x}{1})$ is a finite length $R_{\p}$-module. Let
$$0=C_{0}\subsetneq C_{1}\subsetneq\cdots\subsetneq C_{n}=R_{\p}/({\q}R_{\p},\frac{x}{1})$$ be a composition series of $R_{\p}/({\q}R_{\p},\frac{x}{1})$ as an $R_{\p}$-module. Then $C_{j}/C_{j-1}\cong k(\p)$ for $j=1,2,\cdots,n$. By this we can draw a conclusion that
$$\Tor_{i+1}^{R_{\p}}(k(\p),\Hom_{R}(R_{\p},M))\neq 0.$$
Otherwise, by induction on the length of $R_{\p}/({\q}R_{\p},\frac{x}{1})$, we deduce that
$$\Tor_{i+1}^{R_{\p}}(R_{\p}/({\q}R_{\p},\frac{x}{1}),\Hom_{R}(R_{\p},M))=0.$$
This induces a contradiction.
\end{proof}

\begin{Corollary}\label{Cor:BassPlusii}
Let $R$ be a $U$ ring, $\p,\q\in\Spec~R$ such that $\q\subsetneq\p,~\hgt(\p/\q)=s$, and $M$ an Artinian $R$-module. Then $\pi_{i}(\q,M)\neq 0\Rightarrow\pi_{i+s}(\p,M)\neq 0$.
\end{Corollary}

\begin{proof}
Using induction on $\hgt(\p/\q)$. This Corollary follows from Theorem \ref{Thm:BassPlus}.
\end{proof}

\begin{Corollary}\label{Cor:BassFd}
Let $R$ be a $U$ ring, $\p\in\Spec~R$, $M$ an Artinian $R$-module. If there some $n\in\N$ such that $\pi_{i}(\p,M)=0$ for any $i>n$, then $\fd_{R_{\p}}\Hom_{R}(R_{\p},M)\leq n$.
\end{Corollary}

\begin{proof}
If $\fd_{R_{\p}}\Hom_{R}(R_{\p},M)>n$, then by Lemma \ref{Lem:MinFlatRe} there exists $\q\subseteq\p,~\q\in \Spec~R$, such that $\pi_{n+1}(\q,M)>0$. Then we get $\pi_{\tiny{\hgt(\p/\q)+n+1}}(\p,M)>0$ by Theorem \ref{Thm:BassPlus}. This contradicts to the assumption. Thus $\fd_{R_{\p}}\Hom_{R}(R_{\p},M)\leq n$.
\end{proof}

To obtain the main result of this section, we first prove the following Proposition.

\begin{Proposition}\label{Prop:CogCd}
Let $R$ be a Noetherian topological ring, $M$ a strongly representable linearly compact $R$-module. If $\cograde_{R_{\p}}\Hom_{R}(R_{\p},M)<\infty$, then $$\cograde_{R_{\p}}\Hom_{R}(R_{\p},M)\leq\Cdim_{R_{\p}}\Hom_{R}(R_{\p},M).$$
\end{Proposition}

\begin{proof}
Let $n=\cograde_{R_{\p}}\Hom_{R}(R_{\p},M)$. We use induction on $n$.

For the case $n=0$, this Proposition is obviously true.

Assume that $n>1$, and this Proposition holds for $n-1$. Let $\frac{x}{s}\in\p R_{\p}$ be a $\Hom_{R}(R_{\p},M)$ co-regular element. Then $\frac{x}{s}\in R_{\p}-\bigcup\limits_{\tiny{\q\in\Att_RM,~\q\subseteq \p}}{\q}R_{\p}$, and
$$\cograde_{R_{\p}}(0:_{\tiny{\Hom_{R}(R_{\p},M)}}\frac{x}{s})=\cograde_{R_{\p}}{\Hom_{R}(R_{\p},(0:_{M}x))}=n-1.$$
Hence $x\in R-\bigcup\limits_{\tiny{\q\in\Att_RM,~\q\subseteq \p}}\q$, and by induction hypothesis we get
$$\cograde_{R_{\p}}\Hom_{R}(R_{\p},0:_{M}x)\leq\Cdim_{R_{\p}}\Hom_{R}(R_{\p},0:_{M}x).$$
Since
\begin{eqnarray*}
\Cdim_{R_{\p}}\Hom_{R}(R_{\p},0:_{M}x)
&=&\sup\{~\hgt(\p/\q)~|~\q\in\Cos_R(0:_{M}x),~\q\subseteq\p\}\\
&\leq&\hgt(\p/(\Ann_RM,x))\\
&=&\max\{\hgt(\p/(\q,x))|~\q\in\Att_RM,\q\subseteq\p\}\\
&=&\max\{\hgt(\p/\q)-1|~\q\in\Att_RM,\q\subseteq\p\}\\
&=&\Cdim_{R_{\p}}\Hom_{R}(R_{\p},M)-1,
\end{eqnarray*}
we have $\cograde_{R_{\p}}\Hom_{R}(R_{\p},M)\leq\Cdim_{R_{\p}}\Hom_{R}(R_{\p},M)$.
\end{proof}

The following Theorem is the main result of this paper. We characterize the infimum and supremum of the index of non-vanishing dual Bass numbers of Artinian modules, and study the relations among cograde,co-dimension and flat dimension of co-localization of Artinian modules.

\begin{Theorem}\label{Thm:Main}
Let $R$ be a $U$ ring, $M$ an Artinian $R$-module, $\p\in\Cos_RM$. Then:
\begin{itemize}
\item[$(1)$.] If $\pi_{i}(\p,M)>0$, then $\cograde_{R_{\p}}\Hom_{R}(R_{\p},M)\leq i\leq\fd_{R_{\p}}\Hom_{R}(R_{\p},M)$. If $\cograde_{R_{\p}}\Hom_{R}(R_{\p},M)=s$, then $\pi_{s}(\p,M)>0$. ~($R$ is a Noetherian ring, not necessarily a $U$ ring, and $\fd_{R_{\p}}\Hom_{R}(R_{\p},M)$ is possibly infinite).
\item[$(2)$.] If $\fd_{R_{\p}}\Hom_{R}(R_{\p},M)=t<\infty$,~$\p,\q\in\Spec~R$ such that $\q\subseteq \p$, then $$\pi_{t}(\q,M)>0\Leftrightarrow \q=\p.$$
\item[$(3)$.] If $\fd_{R_{\p}}\Hom_{R}(R_{\p},M)=\infty$. Then for any $n\in\N$, there exists $i\geq n$ such that $\pi_{i}(\p,M)>0$.
\item[$(4)$.] $\cograde_{R_{\p}}\Hom_{R}(R_{\p},M)\leq\Cdim_{R_{\p}}\Hom_{R}(R_{\p},M)\leq\fd_{R_{\p}}\Hom_{R}(R_{\p},M)$.
\end{itemize}
\end{Theorem}

\begin{proof}
$(1)$. By Proposition \ref{Prop:CogBass}, we only need to show $$\pi_{i}(\p,M)>0\Rightarrow i\leq\fd_{R_{\p}}\Hom_{R}(R_{\p},M).$$
For the case $\fd_{R_{\p}}\Hom_{R}(R_{\p},M)=\infty$, there is nothing to prove. Suppose that $\fd_{R_{\p}}\Hom_{R}(R_{\p},M)=r<\infty$. If $i>r$ such that $\pi_{i}(\p,M)>0$, then by Lemma \ref{Lem:MinFlatRe} we get $\pi_{i}(\p,M)=\pi_{i}(\p R_{\p},\Hom_{R}(R_{\p},M))$. Thus
$\fd_{R_{\p}}\Hom_{R}(R_{\p},M)\geq i$. This induces a contradiction. Hence $i\leq\fd_{R_{\p}}\Hom_{R}(R_{\p},M)$.

$(2)$. Since $\fd_{R_{\p}}\Hom_{R}(R_{\p},M)=t<\infty$, then by Lemma \ref{Lem:MinFlatRe} there must exist $\q\subseteq\p,~\q\in\Spec~R$ such that $\pi_{t}(\q,M)>0$. Suppose that $\q\subsetneq \p$. Then  by Corollary \ref{Cor:BassPlusii} we have $\pi_{\tiny{t+\hgt(\p/\q)}}(\p,M)=\pi_{t+\tiny{\hgt(\p/\q)}}(\p R_{\p},\Hom_{R}(R_{\p},M))\neq 0$. This contradict to $\fd_{R_{\p}}\Hom_{R}(R_{\p},M)=t$. Hence $\q=\p$. It is similar to prove the other side.

$(3)$. This conclusion is followed by Corollary \ref{Cor:BassFd}.

$(4)$. By Proposition \ref{Prop:CogCd} we only need to show $$\Cdim_{R_{\p}}\Hom_{R}(R_{\p},M)\leq\fd_{R_{\p}}\Hom_{R}(R_{\p},M).$$

Let $\q$ be a minimal element of $\Cos_RM$ with $\q\subseteq\p$, by \cite[Theorem 4.5]{CuongNhan02i}, we get $\q\in\Att_R(R/\q\otimes_{R}M)\subseteq\Cos_R(R/\q\otimes_{R}M)$. Since
$$\pi_{0}(\q,M)=\dim_{k(\q)}k(\q)\otimes_{R_{\q}}\Hom_{R}(R_{\q},M) =\dim_{k(\q)}\Hom_{R}(R_{\q},R/\q\otimes_{R}M)\neq 0.$$
By Corollary \ref{Cor:BassPlusii}, we get $\pi_{\tiny{\hgt(\p/\q)}}(\p,M)=\pi_{\tiny{\hgt(\p/\q)}}(\p R_{\p},\Hom_{R}(R_{\p},M))\neq 0$.
On the other hand, $\Cdim_{R_{\p}}\Hom_{R}(R_{\p},M)=\max\{~\hgt(\p/\q)~|~\q\in\Cos_RM,~\q\subseteq \p\}$. Hence,
$\Cdim_{R_{\p}}\Hom_{R}(R_{\p},M)\leq\fd_{R_{\p}}\Hom_{R}(R_{\p},M)$.
\end{proof}

\begin{Corollary}Let $R$ be a $U$ ring, $M$ an Artin $R$-module. Then:
\begin{itemize}
\item[$(1)$.] $\Cdim_RM\leq\fd_{R}M$.
\item[$(2)$.] If $\fd_{R}M=r<\infty$, and $\pi_{r}(\p,M)\neq 0$ for some $\p\in\Spec~R$, then $\p\in\Max~R$.
\end{itemize}
\end{Corollary}
\begin{proof} $(1)$. Since $\Cdim_RM=\sup\{~\Cdim_{R_{\m}}\Hom_{R}(R_{\m},M)~|~\m\in\Max~R\}$, by Proposition \ref{Prop:fd} (2) and Theorem \ref{Thm:Main} (4) we have $\Cdim_RM\leq\fd_{R}M$.
%*************************************

$(2)$. If $\pi_{r}(\p,M)\neq 0$, and $\p$ is not a maximal ideal. It follows that there exists a maximal ideal $\m$ such that
$\p\subsetneq\m$, then by Corollary \ref{Cor:BassPlusii} $\pi_{\tiny{r+\hgt(\m/\p)}}(\p,M)\neq 0$. This contradict to $\fd_{R}M=r$.
Hence, $\p$ must be a maximal ideal. %*******************************************************************
\end{proof}


\begin{thebibliography}{99}

\bibitem{Bass63} H. Bass: On the ubiquity of Gorenstein rings, \emph{Math. Z.} 82 (1963), 8-28.

\bibitem{CuongNam01} N. T. Cuong and T. T. Nam: On the co-localization, co-support and co-associated primes of local homology modules, \emph{Vietnam J. Maths.} 29 (2001), 359-368.

\bibitem{CuongNhan02i} N. T. Cuong and L. T. Nhan: On representable linearly compact modules, \emph{Proc. Amer. Math. Soc.} 130 (2002), 1927-1936.

\bibitem{CuongNhan02ii} N. T. Cuong and L. T. Nhan: On the Noetherian dimension of Artinian modules, \emph{Vietnam J. Maths.} 30 (2002), 121-130.

\bibitem{CDN07} N. T. Cuong, N. T. Dung and L. T. Nhan: Top local cohomology and the catenaricity of the unmixed support of a finitely generated module, \emph{Comm. Algebra.} 35 (2007), 1691-1701.

\bibitem{Enochs84} E. Enochs: Flat covers and flat cotorsion modules, \emph{Proc. Amer. Math. Soc.} 92 (1984), 179-184.

\bibitem{EnochsXu97} E. Enochs and J. Z. Xu: On invariants dual to the Bass numbers, \emph{Proc. Amer. Math. Soc.} 125 (1997), 951-960.

\bibitem{EnochsJenda00} E. Enochs and O. M. Jenda: \emph{Relative Homological Algebra}. de Gruyter Expositions in Mathematics. 30 (2000).

\bibitem{Kirby90} D. Kirby: Dimension and length of Artinian modules, \emph{Quart. J. Math. Oxford.} 41 (1990), 419-429.

\bibitem{MelSch95} L. Melkersson and P. Schenzel: The co-localization of an Artinian module, \emph{Proc. Edinburgh Math. Soc.} 38 (1995), 121-131.

\bibitem{Roberts75} R. N. Roberts: Krull dimension for Artinian modules over quasi-local commutative rings, \emph{Quart. J. Math. Oxford.} 26 (1975), 269-273.

\bibitem{TangZakeri94} Z. Tang and H. Zakeri: Co-Cohen-Macaulay modules and modules of generalized fractions, \emph{Comm. Algebra.} 22 (1994), 2173-2204.

\bibitem{Tang96} Z. Tang: Co-Cohen-Macaulay modules and multiplicities for Arinian modules, \emph{J. Suzhou Univ.} (natural science), 12 (1996), 15-26.

\bibitem{X} J. Z. Xu: Minimal injective and flat resolutions of modules over Gorenstein rings, \emph{J. Algebra.} 175 (1995), 451-477.

\end{thebibliography}
\end{document}